\documentclass[12pt]{article}
\usepackage[utf8]{inputenc}
\usepackage[english, russian]{babel}
\usepackage{amssymb,amsfonts,amsthm}
\usepackage{amsmath}

\usepackage{lineno,hyperref}

\usepackage{graphicx}

\usepackage{color}

\usepackage{rotating}
\usepackage{float}

\numberwithin{equation}{section}

\DeclareMathOperator{\Int}{Int}

\newtheorem{theorem}{Theorem}[section]

\newtheorem{thm}{Теорема}[section]
\newtheorem{defin}{Определение}[section]
\newtheorem{lemma}[theorem]{Лемма}
\newtheorem{prop}[theorem]{Предложение}

\newtheorem{remark}[theorem]{Замечание}

\usepackage{babelbib}

\title{Свойства межфазовой границы\\ в параболической задаче с гистерезисом}

\author{Д.Е. Апушкинская, С.Б. Тихомиров, Н.Н. Уральцева} 

\date{\today}

\begin{document}

\maketitle

\begin{abstract}

Исследуются решения параболических уравнений с разрывным оператором гистерезиса, описываемые свободной межфазовой границей. Установлено, что для пространственно трансверсальных начальных данных из пространства $W^{2-2/q}_q$ при $q > 3$ существует решение в пространстве $W^{2, 1}_q$, при этом межфазовая граница обладает гёльдеровой регулярностью с показателем 1/2. Более того для начальных данных из пространства $W^2_{\infty}$ доказано, что межфазовая граница удовлетворяет условию Липшица. Показано, что в случае нетрансверсальных начальных данных   решения с межфазовой границей не существуют.

\noindent
\textit{Ключевые слова: гистерезис, параболическое уравнение, межфазовая граница, трансверсальность, разрешимость.} 
\end{abstract}

\section{Введение}\label{sec:intro}
Рассматривается задача со свободной границей, описываемая уравнением теплопроводности с разрывной правой частью
\begin{equation}
\label{eq:caloric}
\partial_tu-\Delta u=f(u).
\end{equation}
В данной статье мы будем рассматривать нелокальную правую часть \eqref{eq:caloric} $f(u)$, порождаемую разрывным оператором гистерезиса\footnote{ В переводе с греческого ``гистерезис'' означает отставание, запаздывание}.

В простейшей модели гистерезиса заданы два пороговых значения $\alpha <0$ и  $\beta >0$.
Если уравнение \eqref{eq:caloric} описывает тепловой процесс, то 
при $u \geqslant \beta$ возможен только режим охлаждения, 
а при $u \leqslant \alpha$ возможен только режим нагревания. Отметим, что при $\alpha <u(x,t) <\beta$ процесс может находиться в любом из указанных режимов. 
Мы будем называть эти режимы ``фазами'': фазой I при $u \leqslant \alpha$  и фазой II при  $u \geqslant \beta$.  

Фаза II меняется на I, если решение $u$, убывая, достигает порогового значения $\alpha$, и наоборот фаза I меняется на фазу II, если решение $u$, возрастая, достигает порогового значения $\beta$. Таким образом явление гистерезиса заключается в свойстве процессов  (физических, биологических, социологических и т. д.), откликаться на приложенное к ним воздействие в зависимости от их текущего состояния.  Фактически поведение  описываемой системы на интервале времени существенно зависит от её предыстории.

Таким образом, область переменных $(x,t)$, которую мы рассматриваем, разбивается на два заранее неизвестных множества, а  граница раздела фаз является свободной границей. Если бы  
межфазовая граница была известна, то задача стала бы стандартной -- нужно найти решение уравнения теплопроводности с известной ограниченной правой частью и заданными начально-краевыми условиями. Поэтому основная трудность состоит в исследовании характера поведения свободной границы.

Оператор гистерезиса применяется в математических описаниях различных физических, химических и биологических процессов: терморегуляции, химических реакторов, ферромагнетизма, самоорганизации и других (см. монографии \cite{KraP1983, Vis1994, BroSpre1996}). Уравнение \eqref{eq:caloric} с разрывным оператором гистерезиса в правой части было впервые использовано в \cite{HJ1980} при моделировании роста колонии бактерий (Salmonella typhimurium). В работах \cite{HJ1980, HJP1984} был проведен численный анализ построенной модели, однако строгого обоснования предложено не было. 
Далее, аналогичные модели возникали при описании ряда биологических, технологических и химических процессов (см., например, \cite{AU15, CGT2016} и приведенные в них исторические обзоры и библиографию). Из-за разрывной природы гистерезиса вопрос о корректности постановки задачи  \eqref{eq:caloric}  нетривиален. Обычно рассматривается регуляризация оператора гистерезиса многозначным отображением. Это позволяет доказать существование решения \cite{Vis1986,Alt1985,Kop07}; однако вопрос единственности и непрерывной зависимости решения от начальных данных остаётся открытым \cite{Vis2014}.

В работе \cite{GShT13}  для одномерного случая $x \in  \Omega=(0,1)$ на
достаточно малом промежутке времени   $[0,T]$ доказана разрешимость в  $W^{2,1}_q
(Q_T)$, $q>3$, задачи о гистерезисе
\begin{equation}
\begin{aligned}
      \partial_tu-\Delta u&=\mathcal{H}(u) && \text{в}  \quad Q_T=\Omega \times (0,T],\\
      \partial_xu&=0 && \text{на} \ \, S_T=\partial\Omega \times (0,T],\\
      u(x,0)&=\varphi (x) && \text{в}  \quad \Omega,
\end{aligned}
\label{problem_P}
\end{equation}
где известная функция $\varphi \in W^{2-2/q}_q(\Omega)$ удовлетворяет условию трансверсальности (см. ниже)
и  условиям согласования $\partial_x \varphi(0)=\partial_x \varphi(1)=0$, а также полностью задано распределение по фазам: при $\varphi (x)\leqslant \alpha$   процесс
находится в фазе I, при  $\varphi (x)\geqslant \beta$ в фазе  II, а при  $\alpha<\varphi (x)<\beta$ начальные фазы можно задавать произвольно. 
Здесь $\mathcal{H}(u)$ -- значение оператора гистерезиса  на $u(x,t)$, которое в соответствии с описанием данным ранее, однозначно определяется следующим образом.
При $t=0$ распределение фаз нам известно, поэтому положим $\mathcal{H}(u(x,0))=1$ в фазе I и  $\mathcal{H}(u(x,0))=-1$ в фазе II соответственно.
Далее положим\footnote{Наше определение оператора $\cal{H}$ формально отличается от данного в \cite{GShT13}, но, как несложно увидеть, эквивалентно ему.}
\begin{equation}\label{eq:H0pm}
\mathcal{H}(u(x,t))=
\begin{cases}
\ \ 1, & u(x,t) \leqslant \alpha,\\
-1, &u(x,t) \geqslant \beta.
\end{cases}
\end{equation}
Обозначим теперь через $E$ множество
\begin{equation*}
E=\left\{(x,t) \in Q_T: u(x,t) \leqslant \alpha\right\}
\cup \left\{(x,t) \in Q_T: u(x,t) \geqslant \beta\right\}
\cup \left(\Omega \times \{0\}\right).
\end{equation*}
Далее, при $(x,t)\in \overline{Q}_T\setminus E$ положим 
\begin{equation} \label{def-hysteresis}
\mathcal{H}(u(x,t))=\mathcal{H}(u(x,\tau (x))),
\end{equation}
где $\tau(x)=\max\{s \leqslant t: (x,s)\in E\}$. Равенство \eqref{def-hysteresis} означает, что если $u(x,t)\in (\alpha,\beta)$, то оператор гистерезиса принимает  те же значения, что и в ``предыдущий момент'' времени (см. Рис.~\ref{fig:1}).
\begin{figure}
\centering
\includegraphics[scale=0.2]{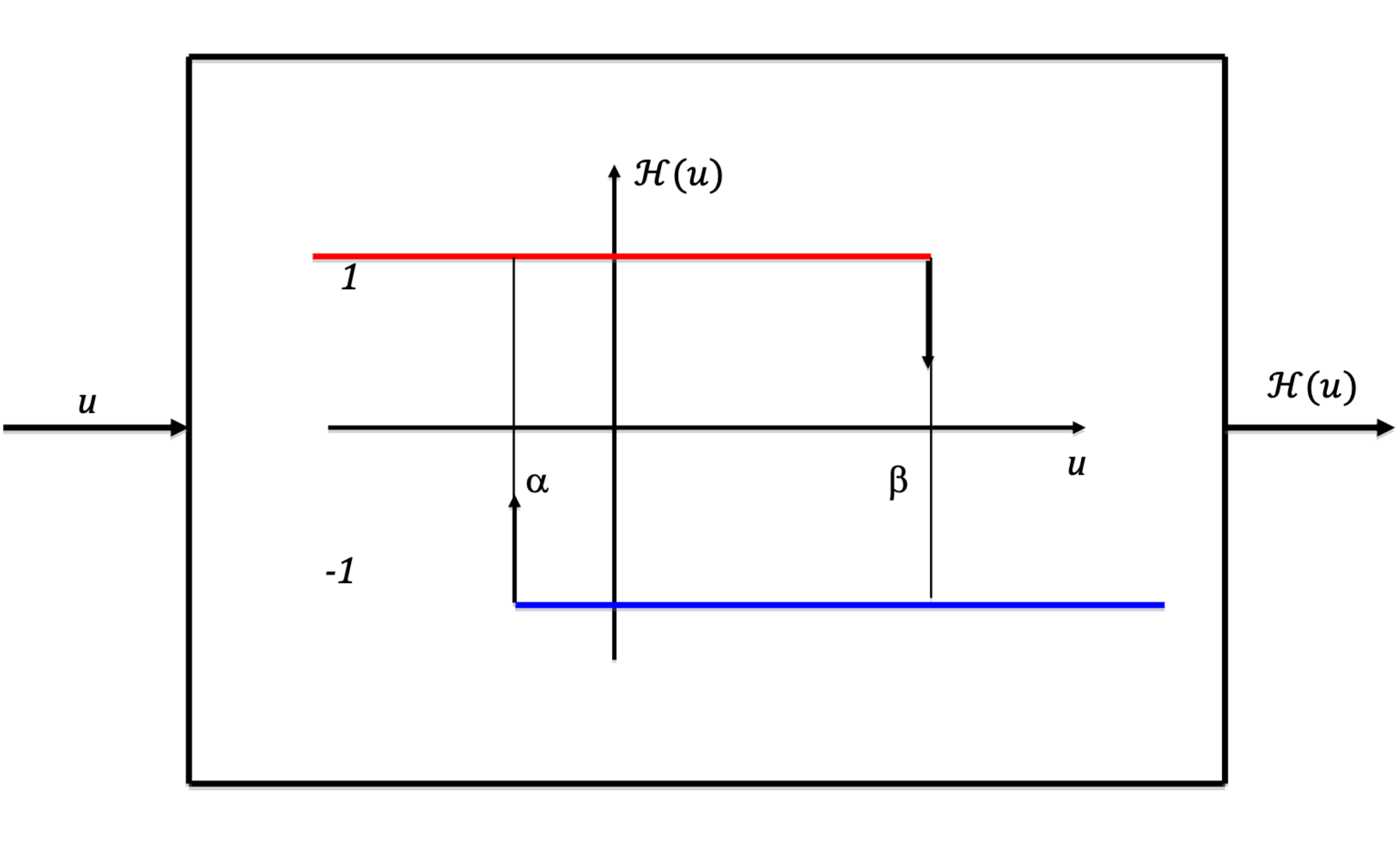} 
\caption{Оператор гистерезиса $\mathcal{H}(u)$.}\label{fig:1}
\end{figure}

Согласно определению, если $\mathcal{H}(u(x,t))=1$, то  $u(x,t)$ находится в фазе I, а если
$\mathcal{H}(u(x,t))=-1$, то  $u(x,t)$ находится в фазе II. При этом предполагается
непрерывность при данном $x$ значения оператора гистерезиса  $\mathcal{H}(u(x,t))$ со стороны б\'{о}льших значений $t$.

Важными являются введенные авторами \cite{GShT13}  условия трансверсальности начальной функции  $\varphi$. Оно состоит в том, что $|\partial_x\varphi (x)|>0$ в тех точках $(x,0)$, где меняется начальная фаза и $\varphi (x) \in \{\alpha, \beta\}$. В \cite{GShT13} было доказано, что если количество точек переключения фаз при $t=0$ конечно и выполнено условие трансверсальности, то каждая ветвь межфазовой границы начинается в некоторой точке при $t=0$  и на некотором малом промежутке времени является графиком гельдеровой функции. Более того, на этом промежутке ветви межфазовой границы не пересекаются. 
Единственность решения задачи \eqref{problem_P} с транверсальными начальными данными была установлена в \cite{GT12}.
В работе \cite{AU15a}, с помощью методов и подходов теории задач со свободными границами, были исследованы свойства локальной регулярности решений  \eqref{problem_P}.

В настоящей работе значительно упрощается по сравнению с \cite{GShT13} 
доказательство существования решения и гельдеровости межфазовых границ, а
также доказывается, что при
$\varphi \in  W ^2_{\infty}(\Omega)$
   каждая ветвь межфазовой границы описывается липшицевой функцией. Конструкция решения будет такая же, как и в \cite{GShT13}, новыми являются
односторонние оценки разностных отношений по  $t$. 

Если начальная функция $\varphi$ не удовлетворяет условию трансверсальности, то вопрос о корректности постановки задачи \eqref{problem_P} остается открытым. В работах \cite{GT2017, GT2018} рассматривалась аппроксимационная задача, основанная на дискретизации пространственной переменной. Было показано, что значения оператора гистерезиса для решения дискретизированной системы образуют сложный паттерн и не имеют предела, когда параметр дискретизации стремится к нулю.  В настоящей работе  мы показываем, что при нетрансверсальной начальной функции не существует решения, описываемого при помощи межфазовой границы. 

Без ограничения общности будем считать, что в точках изменения начальных фаз функция $\varphi$ принимает одно из пороговых значений $\alpha$ или $\beta$. 
Действительно,
если начальная точка межфазовой границы находится далеко от точек, где $\varphi=\alpha$ или $\varphi=\beta$, то на достаточно малом промежутке времени функция $u$ не сможет достигнуть порогового значения. Поэтому граница между фазами будет вертикальной прямой\footnote{Считаем, что ось $0t$ направлена вверх, а ось $0x$ направо.}, так что соответствующая ветвь межфазовой границы будет задаваться гладкой (постоянной) функцией переменной $t$. В этом случае доказательство разрешимости задачи \eqref{problem_P} хорошо известно (см., например, \cite[гл. IV]{LSU67}).

В начальный момент времени распределение фаз задано. В дальнейшем, при $t>0$, расположение фаз, вообще говоря, меняется.  Отметим также, что множества уровня $\{u(x,t)=\alpha\}$ и $\{u(x,t)=\beta\}$ не всегда являются частями межфазовой границы. Действительно, множество уровня $\{u=\alpha\}$
может содержаться внутри фазы I, а множество уровня $\{u=\beta\}$ -- внутри фазы II.
В этом случае межфазовая граница может содержать несколько компонент множеств уровня, соединенные вертикальными отрезками или лучами. Такие вертикальные куски межфазовой границы мы будем называть  ``спящей границей''.  
Пример возможной конструкции межфазовой границы приведен на Рис.~2. 

\begin{figure}

\medskip
\centering
\includegraphics[scale=0.24]{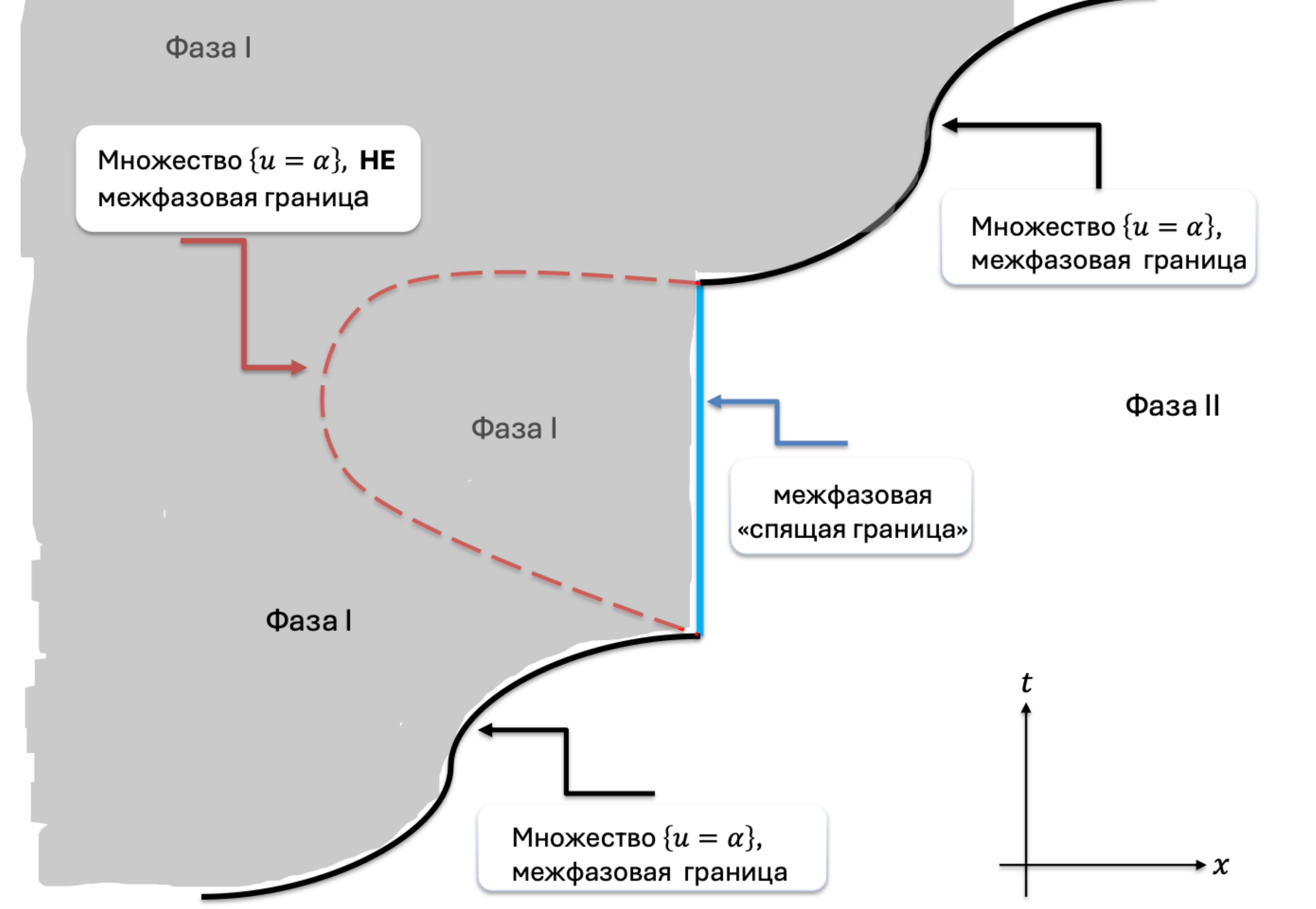} 
\caption{Возможный вариант структуры межфазовой границы.}
\label{fig:2}
\end{figure}

\vspace{0.5cm}

Введем обозначения, которые будут использоваться на протяжении всей статьи:

\noindent $z=(x;t)$ - точка в $\mathbb{R}^{2}_{x,t}$; 

\noindent
$|z|=\sqrt{|x|^2+|t|}$ - параболическое расстояние в $\mathbb{R}^2_{x,t}$;

\noindent
$\partial_x=\dfrac{\partial}{\partial x}$, \quad $\partial_t=\dfrac{\partial}{\partial t}$, \quad $\Delta=\dfrac{\partial^2}{\partial x^2}$;
\vspace{0.2cm}

\noindent
$\Omega=(0,1)$, \quad $Q_T=\Omega \times [0,T]$, \quad $S_T=\partial\Omega \times (0,T]$, \quad $T \leqslant 1$;
\vspace{0.1cm}

\noindent
$v_{+}=\max \{v(x,t),0\}$, \quad $v_{-}=\min \{v(x,t),0\}$. \vspace{0.2cm}

Мы используем стандартные обозначения для функциональных пространств $\mathcal{C}([0,T])$, $\mathcal{C}^1(\overline{\Omega})$ и $\mathcal{C}(\overline{Q}_T)$. Для ограниченной области ${\cal E}\subset\mathbb{R}^{2}_{x,t}$ мы обозначаем через $L^q (\cal{E})$, $1<q\leqslant \infty$, обозначим стандартное пространство Лебега с нормой $\|\cdot\|_{q,\cal{E}}$.

$W^{2,1}_q(Q_T)$ - параболическое пространство Соболева с нормой
$$
\|v\|_{W^{2,1}_q(Q_T)}=\|\partial_tv\|_{q,Q_T}+\|\Delta v\|_{q,Q_T}+\|v\|_{q,Q_T},
$$
а $W^{2-2/q}_q(\Omega)$ и $W^2_{\infty}(\Omega)$ -- пространства следов для $W^{2,1}_q(Q_T)$.

Различные положительные постоянные обознаются через $c$ и $N$ с индексами или без них. Запись $c(\dots)$ означает, что  $c$ зависит только от параметров, указанных в скобках.

\vspace{0.2cm}

Сформулируем ряд известных результатов о линейной задаче, которыми мы будем пользоваться в дальнейшем.

Рассмотрим линейную задачу
\begin{equation}
\label{linear_problem}
\left\{
\begin{aligned}
\partial_t v-\Delta v&=f(x,t) &&\text{в}\quad Q_T,\\
\partial_x v&=0 &&\text{на}\ \  S_T,\\
v(x,0)&=\varphi (x) &&\text{в}\quad \, \Omega,\\
\end{aligned}
\right.
\end{equation}
где  
начальная функция $\varphi$ нам известна, она точно такая же, как и в нелинейной задаче \eqref{problem_P}. 
Мы предполагаем, что функция $f \in L^{\infty}(Q_T)$  удовлетворяет условию 
\begin{equation*}
\sup\limits_{Q_T}|f(x,t)|\leqslant 1.
\end{equation*}
В дальнейшем $f$ будет выбираться специальным образом.

Теория краевых задач для параболических уравнений \cite[гл. IV]{LSU67} гарантирует, что при $1<q<\infty$ единственное решение $v\in W^{2,1}_q(Q_T)$ задачи \eqref{linear_problem} удовлетворяет оценке
\begin{equation}
\label{NN-2}
\|v\|^2_{W^{2,1}_q(Q_T)} \leqslant c(q)\left(\|f\|^2_{q,Q_T}+
\|\varphi \|^2_{W^{2-2/q}_q(\Omega)}\right).
\end{equation}

При $q>3$ из оценки \eqref{NN-2} и теоремы вложения (см., например, \cite[гл. II, лемма~3.1]{LSU67}) 
%Хорошо известно (см. \cite[гл. IV]{LSU67}), что при указанных выше ограничениях на функции $f$ и $\varphi$  существует константа $M>0$ такая, что $\sup\limits_{Q_T}|v|\leqslant M$. Поэтому в силу результатов \cite[\S 8, гл. II и \S 10, гл. III]{LSU67} 
для любых $z_1, z_2\in \overline{Q}_T$ справедливы неравенства 
\begin{align}
|v(z_1)|+|\partial_x v(z_1)| &\leqslant c_0,
\label{NN-3aa}\\
|v(z_1)-v(z_2)| &\leqslant c_1 |z_1-z_2|,
\label{NN-3}\\
|\partial_x v(z_1)-\partial_x v(z_2)| &\leqslant c_2|z_1-z_2|^{\gamma}.
\label{NN-3a}
\end{align}
Здесь $\gamma=\gamma (q) \in (0,1)$, а $c_i$ ($i=0,1,2$) -- абсолютные константы, зависящие только от $q$ и $\varphi$.

Особо подчеркнем, что $c$ и $c_i$ не зависят от $T$. 

Работа организована следующим образом. В $\S 2$   очень подробно
обсуждается случай единственного изменения начальной фазы. 
Далее, в
$\S 3$ мы
рассматриваем   общий случай распределения начальный фаз. Наконец, в $\S 4$ доказывается
липшицева регулярность каждой ветви межфазовой границы в случае, если   $\varphi
\in W^2_{\infty}(\Omega)$. Еще раз отметим, что в \S\S~2-4 рассматриваются трансверсальные начальные данные. В \S \ref{sec:nont} рассматриваются нетрансверсальные начальные данные и показывается, что в этом случае не существует решений с четко определенной межфазовой границей.

\section{Случай одной  ветви межфазовой границы}\label{sec:single}

В этом параграфе мы рассмотрим задачу \eqref{problem_P} с единственной точкой смены начальных фаз.

Исследуем подробно один из возможных частных случаев.
Пусть $b$ -- единственная точка из $\Omega=(0,1)$, в которой у начальной функции $\varphi$ из задачи \eqref{problem_P} при $x<b$ лежит фаза I, а при $x>b$ -- фаза II. Дополнительно мы считаем, что
\begin{equation}
\label{eq:NN_2.5}
\varphi (b)=\alpha
\end{equation}
и выполнено условие трансверсальности
\begin{equation}
\label{eq:NN_2.6}
\partial_x \varphi(b) >0.
\end{equation}
Положим $m=\min\{1,\partial_x\varphi (b)\}.$

Из условия $\varphi \in W^{2-2/q}_q(\Omega)$, $3<q<\infty$, следует, что $\varphi \in {\cal{C}}^1(\overline{\Omega})$. Поэтому можно найти $\sigma>0$ такое, что $[b-\sigma,b+\sigma] \subset (0,1)$ и  справедливо неравенство
\begin{equation}
\partial_x\varphi (x) \geqslant \frac{m}{2}, \quad \text{при}\ |x-b|\leqslant \sigma.
\label{eq:NN_2.7a}    
\end{equation}
Из \eqref{eq:NN_2.7a} очевидно следует
\begin{equation}
\label{eq:NN_2.7}
\varphi (b+\sigma)>\alpha +\frac{m}{2}\sigma, \quad \varphi (b-\sigma)< \alpha -\frac{m}{2}\sigma.
\end{equation}

Поскольку $b$ единственная точка смены фаз, то вне  $\sigma$-окрестности точки $b$ начальная фаза постоянна, т.е.
\begin{equation}\label{eq:NN_2.8a}
\begin{aligned}
&\varphi (x) >\alpha, &&\text{при}\ x\in [b+\sigma,1],\\
&\varphi (x) < \beta, &&\text{при}\ x\in [0,b-\sigma].
\end{aligned}
\end{equation}

Рассмотрим сначала линейную задачу \eqref{linear_problem}, у которой в качестве начального условия выбрана функция $\varphi$ из нелинейной задачи \eqref{problem_P}, удовлетворяющая условиям \eqref{eq:NN_2.5}-\eqref{eq:NN_2.8a}. 

Cогласно \eqref{NN-2}-\eqref{NN-3a} подберем $T>0$ настолько малым, чтобы для любого решения $v\in W^{2,1}_q(Q_T)$ задачи \eqref{linear_problem} при $t\in [0,T]$ были справедливы неравенства
\begin{align}
v (x,t) &> \alpha, &&\text{при}\ x\in (b+\sigma,1]; \qquad 
\label{eq:NN_2.8}\\
 v(x,t) &<\beta, &&\text{при}\ x\in [0,b-\sigma); 
\label{eq:NN_2.8b}\\
\partial_x v(x,t) &\geqslant \frac{m}{4}>0,
&&\text{при}\ |x-b|\leqslant \sigma.
\label{eq:NN_2.10}
\end{align}
Из \eqref{eq:NN_2.10} очевидным образом следует
\begin{equation}
v (b-\sigma,t) \leqslant \alpha -\frac{m}{4}\sigma,
\qquad 
v(b+\sigma,t)\geqslant \alpha+\frac{m}{4}\sigma.
\label{eq:NN_2.8c}\\
\end{equation}
Положительные постоянные $\sigma$ и $T$, зависящие от $\varphi$, зафиксированы и далее меняться не будут. 

Согласно \eqref{eq:NN_2.8c} и \eqref{eq:NN_2.10}
при каждом $t\in [0,T]$ непрерывная и монотонная по $x$ функция $v(x,t)-\alpha$ меняет свой знак на промежутке $|x-b|\leqslant\sigma$. 
Поэтому  для любого $t\in [0,T]$ найдется единственное значение $x=a(t)$ такое, что $|a(t)-b|<\sigma$, $v(a(t),t)=\alpha$ и $a(0)=b$. Другими словами, множество уровня $v(x,t)=\alpha$ при $|x-b|<\sigma$ и $t\in [0,T]$ представляет собой кривую, определяемую явным уравнением $x=a(t)$. 

Покажем, что функция $a(t)$ непрерывна по Гельдеру. Действительно, возьмем произвольные $t_1, t_2 \in [0,T]$ такие, что $t_2>t_1$. По определению функции $a(t)$ выполнено
$$
v(a(t_1),t_1)=v(a(t_2),t_2)=\alpha.
$$
Если $a(t_2)>a(t_1)$, то справедлива оценка
\begin{align*}
\frac{m}{4} \left( a(t_2)-a(t_1)\right) &\leqslant v(a(t_2),t_2)-v(a(t_1),t_2)\\
&=v(a(t_1),t_1)-v(a(t_1),t_2)
\leqslant c_1 |t_1-t_2|^{1/2},
\end{align*}
где в первом неравенстве использована оценка \eqref{eq:NN_2.10}, а в последнем - оценка \eqref{NN-3}.
Если же $a(t_1) > a(t_2)$, то оцениваем так
\begin{align*}
\frac{m}{4} \left(a(t_1) - a(t_2)\right) &\leqslant v(a(t_1),t_1)- v(a(t_2),t_1) \\
&=v(a(t_2),t_2) -v(a(t_2),t_1) \leqslant  
c_1 |t_1-t_2|^{1/2}.
\end{align*} 
Таким образом, в любом случае верна оценка
\begin{equation}
\label{eq:NN-8}
|a(t_1)-a(t_2)|\leqslant \frac{4c_1}{m}|t_1-t_2|^{1/2}.
\end{equation}

Рассмотрим выпуклое замкнутое в $\mathcal{C}([0,T])$ множество $\mathbb{K}$, состоящее из непрерывных 
%по Гельдеру 
монотонно неубывающих функций $\xi: [0,T] \rightarrow \mathbb{R}$, удовлетворяющих условиям  
$$
\xi (0)=b, \qquad \xi(T)-b\leqslant \sigma.
$$

Определим для любой $\xi \in \mathbb{K}$  функцию $f^{\xi}$  следующим образом
\begin{equation}
\label{eq:NN-4}
f^{\xi}(x,t)=\begin{cases}
\ \ 1 &\text{при}\quad x \in (0, \xi(t)],\\
-1 &\text{при}\quad x \in (\xi(t), 1).
\end{cases}
\end{equation}

Далее, мы решаем задачу \eqref{linear_problem} с правой частью $f=f^{\xi}$ и находим функцию $a(t)$ по вышеописанной процедуре.

Установим теперь непрерывную зависимость $a$ от $\xi$ в пространстве $\mathcal{C}\left([0,T]\right)$. Пусть имеются две функции $\xi$ и $\tilde{\xi}$ из множества $\mathbb{K}$ такие, что 
$$
\|\xi -\tilde{\xi}\|_{\mathcal{C}\left([0,T]\right)} \leqslant \varepsilon,
$$
где $\varepsilon$ - произвольная положительная константа.
Поставим им в соответствие решения $v$ и $\tilde{v}$ задач \eqref{linear_problem} с правыми частями $f=f^{\xi}$ и $\tilde{f}=f^{\tilde{\xi}}$, заданными \eqref{eq:NN-4}, соответственно, а также функции $a$ и $\tilde{a}$.
Для разности $v-\tilde{v}$ имеем неравенство
\begin{align*}
\|v-\widetilde{v}\|_{W^{2,1}_q\left(Q_T\right)}&\leqslant
c\|f-\widetilde{f}\|_{q,Q_T}\\
&\leqslant c\left(\int\limits_0^T\int\limits_{\{|\xi(t)<x<\tilde{\xi}(t)|\}\cup\{|\tilde{\xi}(t)<x<\xi|(t)|\}}2^qdxdt\right)^{1/q}\leqslant 2c\varepsilon^{1/q}.
\end{align*}

Учитывая вложение $W^{2,1}_q(Q_T)\hookrightarrow \mathcal{C}\left(\overline{Q}_T\right)$, получаем оценку
$$
\max\limits_{z\in Q_T}|v(z)-\tilde{v}(z)|\leqslant N (q)
\varepsilon^{1/q}.
$$

Поскольку $v(a(t),t)=\tilde{v}(\tilde{a}(t),t)=\alpha$, в случае $a(t)>\tilde{a}(t)$ имеем
$$
\frac{m}{4}\left(a(t)-\tilde{a}(t)\right)\leqslant v(a(t),t)-v(\tilde{a}(t),t)=\tilde{v}(\tilde{a}(t),t)-v(\tilde{a}(t),t)\leqslant N\varepsilon^{1/q}.
$$
Если же $\tilde{a}(t)>a(t)$, то придем к оценке
$$
\frac{m}{4}\left(\tilde{a}(t)-a(t)\right)\leqslant v(\tilde{a}(t),t)-v(a(t),t)=v (\tilde{a}(t),t)-\tilde{v}(\tilde{a}(t),t)\leqslant N\varepsilon^{1/q}.
$$
В результате получаем
\begin{equation}
\label{eq:NN-10}
\|a-\widetilde{a}\|_{\mathcal{C}\left([0,T]\right)}\leqslant \frac{4N}{m}\varepsilon^{1/q}.
\end{equation}

Итак, мы доказали, что функция $a$ непрерывна по Гельдеру и отображение $\xi\mapsto a$ непрерывно в ${\cal{C}}\left([0,T]\right)$. Заметим, что функция $a$, вообще говоря, не является монотонной и по этой причине может не лежать в $\mathbb{K}$.

Введем функцию
\begin{equation}
\label{eq:NN-11}
a_0(t)=\sup\limits_{\tau \in [0,t]}a(\tau),
\end{equation}
которая является верхней монотонной оболочкой для $a$.

Очевидно, что $a_0 \in \mathbb{K}$ 
и представляет собой чередующуюся последовательность замкнутых интервалов строгого возрастания функции $a_0$ и открытых интервалов постоянства (в силу монотонности существует не более чем счетное множество таких интервалов постоянства). На интервалах строгого возрастания $a_0(t)=a(t)$ и, следовательно, $v(a_0(t),t)=\alpha$. 

\begin{lemma}
\label{lemma1}
Пусть $\xi, \tilde{\xi}\in \mathbb{K}$. Предположим, что $a$ и $\tilde{a}$ определяются функциями $\xi$ и $\widetilde{\xi}$ соответственно, а их монотонные оболочки $a_0$ и $\tilde{a}_0$ определяются соотношением  \eqref{eq:NN-11}.

Тогда справедливы следующие утверждения
\begin{align}
&1) &&|a_0(t_2)-a_0(t_1)| \leqslant \frac{4c_1}{m}|t_2-t_1|^{1/2}, \qquad  [t_1,t_2]\subset [0,T], \qquad  \label{eq:NN_2.16}\\
&2) &&\max\limits_{t\in [0,T]}|a_0(t)-\widetilde{a}_0(t)| \leqslant \max\limits_{t\in [0,T]}|a(t)-\widetilde{a}(t)|.
\label{eq:NN_2.17}
\end{align} 
\end{lemma}
\begin{proof}
Докажем утверждение 1). Если $t_1$ и $t_2$ попали на один интервал постоянства, то оценка \eqref{eq:NN_2.16} тривиальна. Иначе,
положим
\begin{align*}
\underline{t_2}&=\max\left\{\tau: \tau \leqslant t_2\ \text{и}\ a_0(\tau)=a(\tau)\right\},\\
\overline{t_1}&=\min\left\{\tau: \tau \geqslant t_1\ \text{и}\ a_0(\tau)=a(\tau)\right\}.
\end{align*}
Ясно, что  $t_1 \leqslant \overline{t_1} \leqslant \underline{t_2} \leqslant t_2$.
Тогда справедлива следующая оценка
\begin{align*}
0 \leqslant a_{0}(t_2)-a_{0}(t_1)=|a(\underline{t_2})-a(\overline{t_1})|
\leqslant \frac{4c_1}{m}\left(\underline{t_2}-\overline{t_1}\right)^{1/2}\leqslant \frac{4c_1}{m}(t_2-t_1)^{1/2},
\end{align*}
где первое неравенство следует из \eqref{eq:NN-8}.

Докажем утверждение 2). Ясно, что $\tilde{a}_0(t)=\tilde{a}(\tau)$ при некотором $\tau \in [0,t]$. Поэтому
$$
\tilde{a}_0(t)=\tilde{a}(\tau)-a(\tau)+a(\tau)\leqslant \tilde{a}(\tau)-a(\tau)+a_0(\tau) \leqslant \tilde{a}(\tau)-a(\tau)+a_0(t),
$$
и, следовательно,
$$
\tilde{a}_0(t)-a_0(t) \leqslant \tilde{a}(\tau)-a(\tau) \leqslant \max\limits_{[0,t]} (\tilde{a}-a).
$$

Меняя ролями $a$ и $\tilde{a}$ и проводя аналогичное рассуждение, заключаем, что
$$
|\tilde{a}_0(t)-a_0(t)|\leqslant \max\limits_{[0,t]}|a-\tilde{a}|\leqslant \max\limits_{[0,T]}|a-\tilde{a}|.
$$
Отсюда  следует искомое неравенство \eqref{eq:NN_2.17}.
\end{proof}

Итак, каждой функции $\xi$ из выпуклого замкнутого множества $\mathbb{K} \subset \mathcal{C}([0,T])$ соответствует функция $a_0\in \mathbb{K}$, причем отображение $a_0=\mathcal{R}(\xi)$  непрерывно в силу \eqref{eq:NN_2.17}, \eqref{eq:NN-10}  и компактно в силу \eqref{eq:NN_2.17}, \eqref{eq:NN_2.16} и теоремы Арцела-Асколи.  Поэтому, в соответствии с принципом Шаудера (см., например, \cite[стр. 628]{KA1984}) существует неподвижная точка $s(t)$ отображения $\mathcal{R}$ и $s(t)\in \mathbb{K}$. 

Заметим также, что если по $s(t)$ построить функцию $f^s(x,t)$
и решить линейную задачу \eqref{linear_problem} с такой правой частью, а затем у линии уровня решения $v(x,t)=\alpha$ взять верхнюю монотонную оболочку, то получим как раз $x=s(t)$. 
Ввиду \eqref{eq:NN_2.8}-\eqref{eq:NN_2.8b} вне отрезка $[b-\sigma, b+\sigma]$ у функции $v(x,t)$ нет точек переключения фаз. Поэтому $f^s(x,t)=\mathcal{H}(v(x,t))$.
Тем самым $u(x,t)=v(x,t)$ является решением нелинейной задачи \eqref{problem_P}, а монотонная кривая $x=s(t)$ -- соответствующей межфазовой границей. 

Строго говоря, кривая $x=s(t)$ представляет собой множество меры нуль и на ней можно задавать $f^s$ произвольным образом, но именно при определении \eqref{eq:NN-4} $f^s(x,\cdot)$ как функция второй переменной оказывается непрерывной справа и является значением пространственно распределенного оператора гистерезиса.

Сформулируем полученный результат.
\begin{thm}\label{thm:1}
Пусть $\varphi \in W^{2-2/q}_q(\Omega)$, $3<q<\infty$, пусть $\partial_x\varphi (0)=\partial_x\varphi (1)=0$ и пусть задано распределение $\varphi$ по фазам $I$ и $II$ такое, что смена фазы происходит только в одной точке $b$. Предположим также, что выполнены условия \eqref{eq:NN_2.5}-\eqref{eq:NN_2.6}.

Тогда существует $T>0$, зависящее от $q$ и $\varphi$, такое, что при $t\in [0,T]$ нелинейная задача \eqref{problem_P} имеет решение $u\in W^{2,1}_q(Q_T)$ и соответствующая межфазовая граница описывается уравнением $x=s(t)$, где $s(t)$ - неубывающая непрерывная по Гельдеру  функция с показателем~$1/2$.
\end{thm}

\begin{remark}
В работе \cite{GShT13} было показано, что свободная граница является Гельдеровой функцией с показателем $\gamma < 1-3/q$. Таким образом при $q<6$ теорема \ref{thm:1} доказывает более высокую регулярность свободной границы.
\end{remark}

\begin{remark}
В теореме \ref{thm:1} мы взяли верхнюю монотонную оболочку функции $a$ в связи с условием \eqref{eq:NN_2.6}. 

При выполнении \eqref{eq:NN_2.5} и условия $\partial_x\varphi(b)<0$  рассматривается нижняя монотонная оболочка $a_0(t)$, определяемая формулой:
$$
a_0(t)=\inf\limits_{\tau\in[0,t]}a(\tau).
$$
Остальные возможные случаи рассматриваются аналогично.
\end{remark}    

\section{Случай нескольких ветвей  межфазовой границы}\label{sec:multy}

Рассмотрим общий случай задачи \eqref{problem_P}.

Пусть задано  начальное значение  $\varphi(x)=u(x,0)$. Будем считать, что $\lim_{x \to 0, 1}$ $\varphi(x) \notin \{\alpha, \beta\}$  и фиксированы
$n$  точек  $0<b_1< \dots<b_n<1$,  в которых (и только в них) происходит чередующаяся смена фаз в начальный
момент времени.
Будем считать, что для любого $i=1,\dots,n$ начальная функция $\varphi (b_i)$  принимает одно из пороговых значений $\alpha$  или  $\beta$, и выполнено условие трансверсальности\footnote{Кроме $b_i$, возможны и другие точки, в которых $\varphi$ принимает пороговые значения, но в них условие трансверсальности может не выполняться, поскольку они не являются точками смены фаз.} 
\begin{equation*}\label{eq:transversality}
|\partial_x\varphi (b_i)| > 0.
\end{equation*}
Это однозначно  определяет  pаспределение  фаз при 
$t=0$,  ибо   в начальный момент при  $\varphi (x) \leqslant \alpha$ 
возможна лишь фаза  I,  а при  $\varphi (x) \geqslant \beta$   возможна  лишь фаза II.

Положим    $m : =  \min\limits_{1\leqslant i \leqslant n} \left\{|\partial_x\varphi  (b_i)|   \right\}$.  Очевидно, что $m> 0$. Не умаляя общности, можно считать $m\leqslant 1$, иначе возьмем $m=1$. 

 Рассмотрим все индексы $i$, для которых  $\varphi (b_i)=\alpha$.  Выберем  $\sigma$ 
так, чтобы одно из чисел  $\varphi (b_i+\sigma)-\alpha$   и   $\varphi (b_i-\sigma) -\alpha$ было  больше   $m \sigma/2$,   а другое меньше    $ - m\sigma /2$.   То же самое  с заменой $\alpha$  на   $\beta$   сделаем в случаях  
$\varphi (b_i)=\beta$. Подчеркнем, что поскольку $\varphi \in W^{2-2/q}_q$, $3<q<\infty$, то мы можем выбрать одно и тоже значение $\sigma$ для всех $b_i$ и сделать $\sigma$ столь малым, чтобы $b_1>\sigma$, $b_n<1-\sigma$, а $b_{i+1}-b_i>2\sigma$ при $1\leqslant i\leqslant n-1$.

Аналогично \S 2, перейдем теперь к линейной задаче \eqref{linear_problem} с начальной функцией $\varphi$ из нелинейной задачи \eqref{problem_P}.

Затем, учитывая \eqref{NN-2}-\eqref{NN-3a},  выберем  $T$   достаточно малым, чтобы для решений $v(x,t)$ задачи \eqref{linear_problem}
при всех $t\in [0,T]$ и при всех индексах $i$ выполнялись условия:
\begin{enumerate}
    \item[1)]  одно из значений $v(b_i+\sigma,t)-\varphi (b_i)$ и $v(b_i-\sigma,t)-\varphi (b_i)$  больше чем $\dfrac{m}{4}\sigma$, а другое меньше чем $-\dfrac{m}{4}\sigma$;
    \item[2)] вне $\sigma$-окрестностей точек $b_i$ функции $v(x,t)$ удовлетворяют неравенству
$v(x,t)>\alpha$, если $\varphi (x)$ находится в фазе II, или неравенству $v(x,t)<\beta$, если $\varphi (x)$ находится в фазе I;
    \item[3)] $|\partial_xv(x,t)|>\dfrac{m}{4}$ при $|x-b_i|\leqslant \sigma$. Кроме того, в $\sigma$-окрестности каждой точки $b_i$ функция
    $\partial_xv(x,t)$ имеет тот же знак, что и $\partial_x \varphi (b_i)$.
\end{enumerate}
В дальнейшем мы считаем положительные постоянные $\sigma$ и $T$ зафиксированными (как и в предыдущем параграфе, их значения зависят только от $\varphi$).

Рассуждая так же как и в \S 2, нетрудно показать, что для всех индексов $i$  линии уровня $v(x,t)=\varphi (b_i)$ при $|x-b_i|<\sigma$ и $t\in [0,T]$ представляют собой кривые $x=a_i(t)$, такие что $a_i(0)=b_i$ и функции $a_i(t)$ удовлетворяют условию Гельдера по переменной $t$ с показателем $1/2$, точнее
\begin{equation*}
    |a_i(t_2)-a_i(t_1)|\leqslant \frac{4c_1}{m}|t_2-t_1|^{1/2}, \qquad \forall t_1,t_2 \in [0,T].
\end{equation*}
 
Будем обозначaть через   $Q^{*}_i $ прямоугольники
$$
Q^{*}_i=\{(x,t): |x-b_i|<\sigma,\ 0\leqslant t \leqslant T\}.
$$

Определим еще промежуточные прямоугольники $Q^{**}_i$,  лежащие между
$Q^{*}_i$ и  $Q^{*}_{i+1}$:  
$$
Q^{**}_i = [b_i+\sigma, b_{i+1}-\sigma] \times [0,T], \qquad
i=1,\dots, n-1. 
$$
Дополним их прямоугольниками   
$$
Q^{**}_0=[0,b_1-\sigma] \times [0,T] \quad  \text{и} \quad
Q^{**}_n=[b_n+\sigma, 1] \times [0,T].
$$

  С каждым  $Q^*_i$  связывается зависящее от $\varphi (b_i)$ и $\partial_x\varphi (b_i)$ выпуклое замкнутое в $\mathcal{C}([0,T])$ множество  $\mathbb{K}_i$, состоящее из монотонных 
 функций $\xi_i (t)$ таких,
что  $\xi_i(0)=b_i$ и                       $|\xi_i(t)-\varphi (b_i)| \leqslant \sigma 
$ при $t\in[0,T]$,
т.е. кривые $x=\xi_i(t)$ содержатся в соответствующих $\overline{Q^*_i}$. Заметим, что характер монотонности функций $\xi_i$ (т.е. невозрастание или неубывание) однозначно определяется значениями $\varphi (b_i)$ и $\partial_x\varphi (b_i)$. Возможные варианты монотонности $\xi_i$ приведены ниже в 
таблице~\ref{tab:1}.

\begin{sidewaystable}
\caption{}
\centering
\begin{tabular}{c|c|c|c}
\ & & & \\
\hline \hline
\ & & & \\
\textbf{Вариант} & \textbf{Поведение} $\varphi$ \textbf{в точке} $b_i$ &
 \textbf{Характер монотонности } & \textbf{Функция $f$ в $Q^*_i$} \\
 \ & & \textbf{$\xi_i(t)$ из $\mathbb{K}_i$} & \ \\
 \ & & & \\
 \hline \hline
 \ & & & \\
 1. &  $\varphi (b_i)=\alpha$ и $\partial_x\varphi (b_i) \geqslant m$ & неубывающие & $f(x,t)=\left\{
\begin{aligned}
&1, &&\text{при}\quad x\leqslant \xi_i(t),\\
-&1, &&\text{при}\quad x > \xi_i(t)
\end{aligned}
\right.
$ \\
\  & & & \\
\hline
\  & &  & \\
2.& $\varphi (b_i)=\alpha$ и $\partial_x\varphi (b_i) \leqslant -m$ & невозрастающие & $f(x,t)=\left\{
\begin{aligned}
&1, &&\text{при}\quad x\geqslant \xi_i(t),\\
-&1, &&\text{при}\quad x < \xi_i(t)
\end{aligned}
\right.
$ \\
\  & & & \\
\hline
\  & &  & \\
3. & $\varphi (b_i)=\beta$ и $\partial_x\varphi (b_i) \geqslant m$ & невозрастающие & $f(x,t)=\left\{
\begin{aligned}
-&1, &&\text{при}\quad x\geqslant \xi_i(t),\\
&1, &&\text{при}\quad x < \xi_i(t).
\end{aligned}
\right.
$ \\
\  & & & \\
\hline
\  & &  & \\
4. & $\varphi (b_i)=\beta$ и $\partial_x\varphi (b_i) \leqslant -m$ &  неубывающие & $f(x,t)=\left\{
\begin{aligned}
-&1, &&\text{при}\quad x\leqslant \xi_i(t),\\
&1, &&\text{при}\quad x > \xi_i(t).
\end{aligned}
\right.
$ \\
\  & &  & \\
\hline
\end{tabular}
\label{tab:1}
\end{sidewaystable}

\vspace{0.2cm}

Дальнейшие рассуждения проводятся
аналогично тому, как это
делалось в \S 2 для случая одной начальной точки смены фаз.

Сначала рассматриваем линейную задачу вида \eqref{linear_problem} с правой частью $f^{\boldsymbol{\xi}}(x,t)$, определяемой набором $\boldsymbol{\xi} (t) =(\xi_1(t),\dots, \xi_n(t))$, $\xi _i \in \mathbb{K}_i$, $\xi _i(0) = b_i$  и начальной функцией $\varphi$ следующим образом. В каждом прямоугольнике 
$Q^*_{i}$ функцию $f^{\boldsymbol{\xi}}(x,t)$ определяем согласно таблице~\ref{tab:1}. Далее, в прямоугольниках $Q^{**}_i$ считаем $f^{\boldsymbol{\xi}}(x,t) \equiv 1$, если в основании $Q^{**}_i$ (т.е. при $t=0$) функция $\varphi$ находится в фазе I, и полагаем $f^{\boldsymbol{\xi}}(x,t) \equiv -1$, если в основании $Q^{**}_i$ функция $\varphi (x)$ 
находится в фазе II.

Решаем задачу \eqref{linear_problem} с такой функцией $f^{\boldsymbol{\xi}}(x,t)$. Из \eqref{NN-3} следует, что кривые $v(x,t)=\varphi (b_i)$ или, что то же самое, кривые $x=a_i(t)$ гельдеровы по $t$ с показателем $1/2$. Кроме того, отображения $\xi_i \mapsto a_i$ непрерывны в $\mathcal{C}\left([0,T]\right)$. В силу условия трансверсальности, в прямоугольниках $Q_i^{*}$ справедливы неравенства
$$
\partial_x v \geqslant \frac{m}{4} \quad \text{при реализации вариантов 1 и 3}
$$
и
$$
\partial_x v \leqslant -\frac{m}{4} \quad \text{при реализации вариантов 2 и 4}.
$$ 
Как и ранее, функции $a_i(t)$ не обязаны быть монотонными и, следовательно,  содержаться в соответствующих $\mathbb{K}_i$. Поэтому для каждого $i=1,\dots,n$, рассмотрим вместо функции $a_i(t)$ ее монотонную оболочку (верхнюю в случаях 1 или 3, и нижнюю в случаях 2 или 4) $a_{i,0}(t)$  лежащую в $\mathbb{K}_i$.

Аналогично рассуждению из \S2, можно показать, что функции $a_{i,0}$, $i=1,\dots,n$, удовлетворяют условию Гельдера с показателем $1/2$ на промежутке $[0,T]$. 

Введем теперь в рассмотрение банахово пространство $\mathcal{B}=\left({\cal{C}}([0,T])\right)^n$ -- декартово произведение $n$ пространств (слоев) $\mathcal{C}\left([0,T]\right)$. 
Норма элемента $\mathbf{w}=(w_1,\dots, w_n) \in \mathcal{B}$ определяется формулой
$$
\|\mathbf{w}\|_{\mathcal{B}}=\max\limits_{1\leqslant i \leqslant n}\|w_i\|_{{\cal{C}}([0,T])}.
$$
Из определения пространства $\mathcal{B}$ следует, что утверждение о выпуклости (замкнутости, компактности) множества $E=E_1\times \dots \times E_n$ в пространстве $\mathcal{B}$ равносильно выпуклости (замкнутости, компактности) каждого из множеств $E_i$  в пространстве $\mathcal{C}([0,T])$. 

Поскольку каждой функции $\xi_i(t)\in \mathbb{K}_i$, $i=1,\dots,n$, сопоставляется функция $a_{i,0}(t)\in \mathbb{K}_i$, определен оператор 
$$
\mathcal{R}: \mathcal{B} \rightarrow \mathcal{B}, \qquad \mathcal{R} (\boldsymbol{\xi}(t))=(a_{1,0}(t), \dots, a_{n,0}(t)).
$$  
На каждом слое оператор $\mathcal{R}$ непрерывен, компактен и переводит  $\mathbb{K}_i$ в себя.
Таким образом, множество $\mathbb{K}_1\times \mathbb{K}_2\times \dots \times \mathbb{K}_n$ переходит в себя  и к оператору $\mathcal{R}$ применим принцип Шаудера о неподвижной точке.

Неподвижная точка $\mathbf{s}(t)=(s_1(t), \dots, s_n(t))$ оператора $\mathcal{R}$ -- это набор $n$ монотонных ветвей межфазовых границ нелинейной задачи \eqref{problem_P}. При этом, каждая ветвь $x=s_i(t)$ представляет собой чередующуюся последовательность интервалов строгой монотонности и интервалов постоянства. На интервалах строгой монотонности выполняется равенство $v(s_i(t),t)=\varphi (b_i)$. Соответствующее решение $v(x,t)$ совпадает с решением $u(x,t)$ нелинейной задачи \eqref{problem_P}.

Сформулируем теперь окончательный результат
\begin{thm}
\label{theorem2}
Пусть $\varphi \in  W^{2-2/q}_q(\Omega)$ при
$q>3$ и выполнено  $\partial_x\varphi(0)=\partial_x\varphi(1)=0$.

Предположим, что даны несколько различных точек $b_i$, $i=1,\dots,n$, расположенных внутри  $\Omega$, где происходит смена фаз
в начальный момент,  и $\varphi$ принимает в каждой из этих точек одно из пороговых значений $\alpha$ или $\beta$. Кроме того,  в точках $b_i$ выполнены условия трансверсальности, а также
необходимое условие очередности начальных фаз.

Тогда
существует  $T>0$, зависящее от $q$ и   $\varphi$, такое, что задача  \eqref{problem_P} при   $t \in [0,T]$   имеет решение
 $u \in W^{2,1}_q(Q_T)$, а межфазовые кривые
определяются   монотонными гельдеровыми функциями $x=s_i(t)$, $i=1,\dots,n$.
\end{thm}

\section{О липшицевости межфазовой границы}\label{sec:lipschitz}

Выше было доказано существование решения задачи \eqref{problem_P}, в том числе установлена
гельдеровость межфазовой границы. Мы покажем, что  при дополнительном
условии  $\varphi  \in W^2_{\infty} (\Omega)$,  каждая ветвь межфазовой границы  липшицева.   

Начнем с доказательства
простого, но весьма полезного утверждения о решениях линейной задачи \eqref{linear_problem}.

\begin{lemma}
\label{lemma3}
Предположим, что  $v(x,t)$   -  решение задачи \eqref{linear_problem}, и пусть
$\varphi  \in W^2_{\infty} (\Omega)$. 
Существует постоянная  $N_0=N_0(\varphi)>0$ 
такая, что
\begin{equation}
\label{NN-0823-1}
|v(x,t) - \varphi (x)| \leqslant N_0t.      \end{equation}
\end{lemma}

\begin{proof}
Рассмотрим функцию
$$
w(x,t)=t\pm \varepsilon \left(v(x,t)-\varphi (x)\right).
$$
Легко видеть, что 
\begin{equation}
\label{NN-4.2}
\partial_tw -\Delta w=1\pm \varepsilon \left(f+\Delta\varphi\right)\geqslant0
\end{equation}
при достаточно малом $\varepsilon>0$.

Домножим правую и левую части \eqref{NN-4.2} на $w_-$  и проинтегрируем по $\Omega \times [0,t]$, $t\leqslant T$. Интегрируя по частям с учетом условий $w(x,0)=0$ и $\partial_xw \bigg|_{S_T}=0$, получим
$$
\frac{1}{2}\int\limits_{\Omega}w_-^2(x,t)dx+
\int\limits_{\Omega}\int\limits_0^t
\partial_x w_-^2(x,t)dtdx \leqslant 0,
$$
откуда $w_-\equiv 0$.

Итак, мы установили неравенство \eqref{NN-0823-1} с  $N_0= \varepsilon^{-1}$.
Можно, например, положить  $N_0=\sup\limits_{ \Omega} 
|\Delta \varphi(x)|$.
\end{proof}

Далее мы считаем, что выполнены условия теоремы~\ref{theorem2} и рассматриваем разностные отношения по переменной $t$ для 
решения  задачи \eqref{problem_P},  которые определяются
следующим образом
$$
u_{(h)} (x,t)=\frac{u(x,t+h)-u(x,t)}{h},   \quad h>0.
$$
Чтобы они были определены для некоторого диапазона значений  $h$, например при
$h\in (0,\epsilon]$, $\epsilon <T$, придется сократить  высоту рассматриваемых прямоугольников.

В частности, из \eqref{NN-2} вытекает, что при любом $ h \in (0,\epsilon]$ справедливо
неравенство
\begin{equation}
\label{NN-0823-2}
   \|u_{(h)}\|_{q,Q_{T-\epsilon}} \leqslant N_1,   
\end{equation}
поскольку возможно продолжение  $u(x,t)$   функцией   $\varphi (x)$  на  $t<0$  и
для продолженной функции сохраняется оценка  $L_q$-нормы $\partial_tu$.

Важным следствием неравенства \eqref{NN-0823-1} является оценка
\begin{equation}
\label{NN-0823-3}
    \max\limits_{x\in \Omega} |u_{(h)}(x,0)| \leqslant N_0,  \quad   h \in (0,\epsilon].  
\end{equation}

Попытка получить оценку  $|\max\limits_{\Omega} u_{(h)}(x,t)|$   при  $t>0$  была бы неудачной, 
поскольку    $u_{(h)}$   подчиняется соотношениям
\begin{equation}
\label{NN-0823-4}
\begin{aligned}
    \partial_t u_{(h)}-\Delta u_{(h)}&=f_{(h)} &&\text{в} \quad  Q_{T-\epsilon},\\                             
       \partial_x u_{(h)}&=0  &&\text{на}\ \, S_{T-\epsilon},\\
       |u_{(h)}(x,0)| &\leqslant N_0  &&\text{в} \quad \Omega,
\end{aligned}      
\end{equation}
где правая часть в уравнении \eqref{NN-0823-4} становится сингулярной в
окрестности межфазовой границы при  $h \to  0$.

Распределение фаз  при  $t=0$ 
определено однозначно функцией  $\varphi (x)$. В прямоугольниках\footnote{Мы сохраним прежние обозначения из \S3, так что теперь
$Q^*_i$ - прямое произведение $\sigma$-окрестности точки $b_i$ на  $[0,T-\epsilon]$ и т. д.}  $Q^{**}_i$, $i=0, \dots, n$, 
на всем промежутке   $[0,T-\epsilon]$  фаза сохраняет то же значение, которое было
при   $t=0$. 

\begin{lemma}
\label{@Nazarov}
\begin{equation}
\label{NN-0823-5}
   \max\limits_{Q^{**}_i}|u_{(h)}(x,t)| \leqslant N_2,  \quad i=0,\dots,n, 
\end{equation}
где постоянная  $N_2$  определяется $q$ и функцией  $\varphi$.
\end{lemma}

\begin{proof}
Положим $$
\rho=\frac{m \sigma}{8 \max\limits_{Q_T} |\partial_xu|}<\frac{\sigma}{8}
$$ 
и рассмотрим расширенные промежуточные прямоугольники
\begin{gather*}
Q^{**}_{i,\rho}=[b_i+\sigma-\rho, b_{i+1}-\sigma+\rho]\times [0,T-\epsilon], \qquad i=1,\dots,n-1,\\
Q_{0,\rho}^{**}=[0,b_1-\sigma+\rho]\times [0,T-\epsilon], \qquad Q_{n,\rho}^{**}=[b_n+\sigma-\rho,1]\times[0,T-\epsilon].
\end{gather*}
Из доказательства теоремы~\ref{theorem2} следует, что при всех $i=0,\dots,n$ расширенные
прямоугольники 
$Q^{**}_{i,\rho}$ не пересекают межфазовые границы.

%Введем гладкие срезающие функции $\zeta_i=\zeta_i (x)$, такие, что $\zeta_i \equiv 1$  в $Q^{**}_i$ и $\zeta_i \equiv 0$ вне $Q^{**}_{i,\rho}$. 
%\begin{align*}
%\partial_t (u_{(h)}\zeta_i)-\Delta (u_{(h)}\zeta_i)&= f_{(h),i} \quad \text{в}\quad Q^{**}_{i,\rho},
%\end{align*}
%где $$
%\|f_{(h),i}\|_{q,\,Q^{**}_{i,\rho}}\leqslant c(q,\rho)\left(\|u_{(h)}\zeta_i\|_{q,\,Q^{**}_{i,\rho}}+ \|\partial_x(u_{(h)}\zeta_i)\|_{q,\,Q^{**}_{i,\rho}}\right).
%$$
%На боковых границах $Q^{**}_{i,\rho}$ функции $u_{(h)}\zeta_i$ удовлетворяют однородному условию Неймана.

Представим $u_{(h)}$ в виде суммы $w_1+w_2$, где $w_1$  -- решение задачи
$$
\partial_t w_1 -\Delta w_1= 0 \quad \text{в}\quad Q^{**}_{i,\rho}, \qquad w_1\big|_{t=0}=u_{(h)}\big|_{t=0},
$$
постоянное на боковых границах $Q^{**}_{i,\rho}$.  По принципу максимума из \eqref{NN-0823-3} следует оценка
\begin{equation}
    \label{AN-1}
\max\limits_{Q^{**}_i}|w_1(x,t)| \leqslant N_0,  \quad i=0,\dots,n. 
\end{equation}
Поскольку $u_{(h)}$ в $Q^{**}_{i,\rho}$ удовлетворяет однородному уравнению теплопроводности, имеем
\begin{equation*}
\partial_t w_2 -\Delta w_2= 0 \quad \text{в}\quad Q^{**}_{i,\rho}, \qquad w_2\big|_{t=0}=0.
\end{equation*}
Поэтому функция $w_2$, продолженная нулем при $t<0$ также удовлетворяет однородному уравнению теплопроводности.
По теореме об оценке производных решения гипоэллиптического уравнения (см., например, \cite[Предложение 6.6]{Shu2003}) получим
$$
\max\limits_{Q^{**}_i}|w_2(x,t)| \leqslant c(q,\rho)\|w_2\|_{q,Q^{**}_{i,\rho}},  \quad i=0,\dots,n. 
$$
С учетом \eqref{NN-0823-2} и \eqref{AN-1} это дает \eqref{NN-0823-5}.
\end{proof}

Не умаляя общности, можно считать, что
$N_2 \geqslant N_0$.

Рассмотрим теперь какой-либо из прямоугольников $Q_i^*$, в которых происходит смена фаз. Как было доказано в \S3, межфазовые границы задаются монотонными функциями $x=s_i(t)$. Поэтому на каждой прямой параллельной оси  $0t$,
функция $f$ меняет знак не более одного раза. Следовательно, правая часть в первом из уравнений \eqref{NN-0823-4}  имеет определенный знак в каждом из прямоугольников $Q_i^*$,
что позволяет получить \textbf{\textit{односторонние}} оценки разностных отношений  $u$  по  
$t$. 

\begin{lemma}
\label{lemma4}
Для решений $u$ нелинейной задачи \eqref{problem_P} справедливы оценки
\begin{equation}
\label{NN-8-email}
\begin{aligned}
        \min\limits_{Q^*_i} u_{(h)} &\geqslant -N_2, &&\text{если} \ \varphi (b_i)=\alpha, \\  
        \max\limits_{Q^*_i} u_{(h)} &\leqslant N_2, &&\text{если} \ \varphi (b_i)=\beta,
\end{aligned}
\end{equation}
где $h$ -- произвольный параметр из $(0,\epsilon]$. 
\end{lemma}

\begin{proof}
Заметим, что из \eqref{NN-0823-5} следует, что на параболической границе
любого из  цилиндров $Q^*_i$ выполнено  неравенство  $\max |u_{(h)}| \leqslant N_2$. Далее, из таблицы~1 видно, что если $\varphi (b_i)=\alpha$, то $f_{(h)}\geqslant 0$ в $Q_i^*$, а если 
$\varphi (b_i)=\beta$, то $f_{(h)}\leqslant 0$ в $Q_i^*$.

Для получения первого из неравенств \eqref{NN-8-email}, аналогично доказательству леммы~\ref{lemma3}, умножим первое из уравнений
 \eqref{NN-0823-4}  на  $\left(u_{(h)}+N_2\right)_-$,  проинтегрируем по множеству $(b_i-\sigma, b_i+\sigma)\times (0,t)$, $t \leqslant T-\epsilon$, и произведем интегрирование по частям. 

Аналогично выводится второе из неравенств  \eqref{NN-8-email}.
\end{proof}

Теперь мы будем использовать совместно лемму~\ref{lemma4} и неравенства, вытекающие
из условия
трансверсальности. Пусть $\varphi (b_i)=\alpha$ и $\partial_x \varphi (b_i) \geqslant m$. Рассмотрим прямоугольник   $Q^*_i$ и   в нем
неубывающую межфазовую границу  $x=s_i(t)$,   выберем произвольные   $t^*<t^{**}$ 
из  $(0,T-\epsilon]$, в которых функция   $s_i(t)$  строго возрастает (напомним, что тогда $u(s_i(t^*),t^*)=u(s_i(t^{**}),t^{**})=\alpha$). Оценим 
разность  $s_i(t^{**})-s_i(t^*)$.  
В силу леммы~\ref{lemma4}  имеем
$$
u(s_i(t^{**}),t^*)-\alpha=u(s_i(t^{**}),t^*)-u(s_i(t^{**}),t^{**}) \leqslant N_2
\left(t^{**} -t^*\right).
$$
С другой стороны, в силу условия трансверсальности  
$$
u(s_i(t^{**}),t^*)-\alpha=u(s_i(t^{**}),t^*)-u(s_i(t^*),t^*)
\geqslant \frac{m}{4} \left(s_i(t^{**})-s_i(t^*)\right),
$$  
откуда 
$$
0 \leqslant s_i(t^{**})-s_i(t^*) \leqslant
\frac{4 N_2} {m} \left(t^{**}-t^*\right). 
$$
В последнем неравенстве значения   $t^*$  и   $t^{**}$  можно уже брать произвольными из  
$(0,T-\epsilon]$, поскольку на каждом из оставшихся кусков 
 межфазовой  кривой  приращение функции   $s_i(t)$  
равно нулю. 

Пусть теперь $\varphi (b_i)=\alpha$ и $\partial_x \varphi (b_i) \leqslant -m$. Тогда для произвольных $t^*<t^{**}$  из множества строгого убывания функции   $s_i(t)$ с помощью леммы~\ref{lemma4} и  условия трансверсальности доказывается, что
$$
  0 \geqslant s_i(t^{**})-s_i(t^*)\geqslant  -\frac{4N_2} {m} 
  \left(t^{**}-t^*\right).
$$
Далее, легко видеть, что это неравенство верно при любых   $t^*<t^{**}$   из  $(0,T-\epsilon]$. 
Поэтому в обоих случаях справедлива оценка
\begin{equation}
\label{NN-0823-7}
                  |s_i(t^{**})-s_i(t^*)| \leqslant \frac{4N_2} {m} \left(t^{**}-t^*\right).  
\end{equation}
Аналогично
устанавливается неравенство \eqref{NN-0823-7} в прямоугольниках $Q^*_i$ при условии $\varphi (b_i)=\beta$. 

Поскольку константа Липшица в \eqref{NN-0823-7} не зависит от выбора $\epsilon$, мы можем устремить $\epsilon$ к нулю. Это дает
следующий результат.

\begin{thm}
\label{theorem3}
Предположим, что выполнены условия теоремы~\ref{theorem2}  и кроме того $\varphi \in W^2_{\infty}
(\Omega)$.  

Тогда для любого фиксированного $q$,  $3<q<\infty$, решение задачи \eqref{problem_P}, полученное в теореме~\ref{theorem2}, принадлежит $W^{2,1}_q(Q_T)$.\footnote{Напомним, что $T$, вообще говоря, зависит от $q$.} При этом, функции $s_i(t)$, $i=1,\dots,n$, задающие ветви межфазовой границы удовлетворяют условию Липшица с константой $\frac{4N_2}{m}$.
\end{thm} 

\section{Случай нетрансверсальных начальных данных} \label{sec:nont}

Ниже мы покажем, что в окрестности точки, в которой начальная функция $\varphi$ не удовлетворяет условию трансверсальности, межфазовая граница  не является графиком гельдеровой (и даже непрерывной) функции.

\begin{defin}
Будем \ говорить, \ что\  начальная \ функция \ $\varphi \in W^{2-2/q}_q(\Omega)$, $q>3$, из задачи \eqref{problem_P} \textit{топологически нетрансверсальна} в точке $x_0 \in \Omega$,  если для некоторой окрестности $(x_0-\varepsilon,x_0+\varepsilon)$ выполнено
одно из двух  условий:
\begin{itemize}
    \item процесс находится в фазе I при $x \in (x_0-\varepsilon,x_0)\cup(x_0,x_0+\varepsilon)$ и
\begin{equation}\label{eq:phix0}
  \varphi(x_0) = \beta,  \quad \varphi(x) < \beta \quad \text{при}\quad x \in (x_0-\varepsilon,x_0)\cup(x_0,x_0+\varepsilon);
\end{equation}
См. Рис. \ref{fig:nt}.
\begin{figure}[ht]
\centering
\includegraphics[width = 0.5\textwidth]{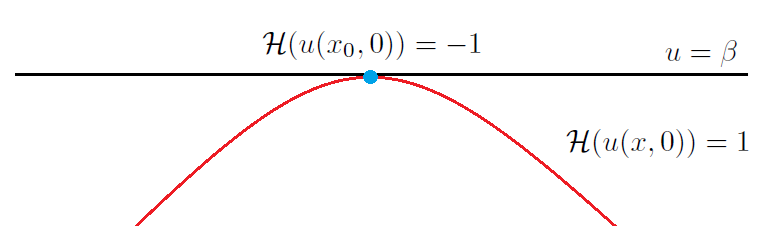} 
\caption{Пример нетрансверсальных начальных данных.}
\label{fig:nt}
\end{figure}

\item процесс находится в фазе II при $x \in (x_0-\varepsilon,x_0)\cup(x_0,x_0+\varepsilon)$ и
\begin{equation}\notag
  \varphi(x_0) = \alpha,  \quad \varphi(x) > \alpha\quad \text{при}\quad x \in (x_0-\varepsilon,x_0)\cup(x_0,x_0+\varepsilon).
\end{equation}
\end{itemize}
\end{defin}

\begin{remark}
    Отметим, что если начальная функция $\varphi$ топологически нетрансверсальна в точке $x_0$, то $\partial_x \varphi(x_0) = 0$, и следовательно $\varphi$ не удовлетворяет условиям трансверсальности из \S\S~\ref{sec:single}-\ref{sec:multy}.
\end{remark}

Предположим, что существует решение $u\in W^{2, 1}_q(Q_T)$ задачи \eqref{problem_P} с топологически нетрансверсальной начальной функцией $\varphi \in W^{2-2/q}_q(\Omega)$, $q>3$. Не умаляя общности, предположим, что выполнено условие \eqref{eq:phix0}, тогда  для некоторых $\varepsilon_1 < \varepsilon$ и $T_1 < T$ выполнено
\begin{equation}\label{eq:ualpha}
u(x, t) > \alpha, \quad x \in (x_0-\varepsilon_1,x_0+\varepsilon_1), \quad t \in (0, T_1);
\end{equation}
\begin{equation}\label{eq:ub}
u(x_0 \pm \varepsilon_1, t) < \beta, \quad t \in (0, T_1).
\end{equation}
Обозначим ${\mathcal{I}} = (x_0-\varepsilon_1,x_0+\varepsilon_1)$. Из \eqref{eq:ualpha} следует, что для любого $x \in {\mathcal{I}}$ фаза решения на интервале $t \in (0, T_1)$ меняется не более одного раза. При этом смена фазы в точке $x\neq x_0$ происходит в момент
\begin{equation}\label{eq:r}
r(x) = \inf_{t \in [0, T_1)}\{u(x, t) \geqslant \beta \}.
\end{equation}
В случае, если $u(x, t) < \beta$ при $t \in [0, T_1)$, положим $r(x) = T_1$. 

Введем обозначения
$$
H^+(\tau) = \{(x, t) \in {\mathcal{I}} \times (0, \tau): \; {\mathcal{H}}(u(x, t)) = 1\}, \quad \tau \in (0, T_1],
$$
$$
H^-(\tau) = \{(x, t) \in {\mathcal{I}} \times (0, \tau): \; \mathcal{H}(u(x, t)) = -1\}, \quad \tau \in (0, T_1],
$$
$$
H^+ = H^+(T_1), \quad H^- = H^-(T_1).
$$

 Отметим, что поскольку $u$ непрерывна, то $r$ полунепрерывна снизу и выполнены равенства
$$
H^+(\tau) = \{(x, t) \in \mathcal{I} \times (0, \tau): t < r(x)\},
$$
$$
H^-(\tau) = \{(x, t) \in \mathcal{I} \times (0, \tau): t \geqslant r(x)\}.
$$

Множество точек $(x, r(x))$, $x \in \mathcal{I}$ образует  границу между фазами I и II. Покажем, что она не может иметь вид из \S\S~\ref{sec:single}-\ref{sec:multy}. Отметим, что на межфазовой границе типа линии уровня $u(x, t) = \beta$, см. Рис.~\ref{fig:2}, функция $r$ непрерывна, а ``спящей границе'' соответствует разрыв первого рода функции $r$.

\begin{thm}\label{thm:noFB}
    Рассмотрим $u \in W_q^{2, 1}(Q_{T_1})$ решение задачи \eqref{problem_P} с топологически нетрансверсальной начальной функцией $\varphi \in W^{2-2/q}_q(\Omega)$, удовлетворяющей условию \eqref{eq:phix0}.  
    
    Предположим, что на некотором интервале $(a, b)$ функция $r$, определенная в \eqref{eq:r}, удовлетворяет неравенству $r(x) < T_1$. 
    
    Тогда
    \begin{itemize}
        \item[(i)] если $r$ непрерывна на $(a, b)$, то  она постоянна на $(a, b)$;
        \item[(ii)] $\limsup_{x \to x_1-}r(x) = \limsup_{x \to x_1+}r(x)$ для любого $x_1 \in (a, b)$.
    \end{itemize}
\noindent
Более того, функция $\tilde{r}$ определенная равенством
        $$
            \tilde{r}(x_1) = \limsup_{x \to x_1}r(x), \quad x_1 \in (a, b)
        $$
   постоянна, и $u(x_1,\tilde{r}(x_1)) = \beta$.
\end{thm}

Пункты $(i)$ и $(ii)$ показывают, что у задачи \eqref{problem_P} не может быть части свободной границы, описываемой аналогично \S~\ref{sec:single}. Действительно, в \S~\ref{sec:single} часть границы, лежащая на линии уровня $u=\beta$, задается  равенством $x = a(t)$ со строго монотонной непрерывной функцией $a$, а следовательно, может быть описана равенством $t=r(x)$  с непрерывной функцией $r$, что невозможно согласно пункту $(i)$. Аналогично, ``спящая граница'' невозможна согласно пункту $(ii)$. Мы предполагаем, что задача \eqref{problem_P} с нетрансверсальной начальной функцией не имеет решения, но не можем исключить существования патологического решения, у которого множество точек переключения гистерезиса имеет положительную меру и нигде не плотно (в этом случае теорема ~\ref{thm:noFB} неприменима) или для некоторых $t_0$ выполнено $u(x, t_0) = \beta$ на открытом интервале $x \in (a, b)$ (что соответствует последнему утверждению теоремы~\ref{thm:noFB}).

В доказательстве теоремы~\ref{thm:noFB} мы будем использовать следующие версии сильного принципа максимума \cite[гл.~3]{krylov1985} и леммы о нормальной производной (см., напр., \cite[Theorem~3.1]{AN2019} или \cite{N2012}). 

\begin{prop}\label{lem:max}
    Предположим, что в некотором прямоугольнике $K = (x_1, x_2) \times (t_1, t_2]$ функция $u(x, t) \in W^{2, 1}_q(K)$, $q>3$, удовлетворяет неравенствам
    $$
    \partial_t u(x, t)-\Delta u(x, t) \leqslant 0, \quad \mbox{для $(x, t) \in K$}
    $$
    и при некотором $M \in \mathbb{R}$
    $$
    u(x, t) \leqslant M, \quad \mbox{для $(x, t) \in \overline{K}$}.
    $$
    Тогда если для некоторого $(x',t') \in K$ выполнено $u(x',t') = M$, то
    $$
        u(x,t) = M, \quad \mbox{для $(x, t) \in (x_1, x_2)\times (t_1, t')$}.
    $$
\end{prop}
\begin{prop}\label{lem:hopf}
    Предположим, что в некотором прямоугольнике $K = (x_1, x_2)\times (t_1, t_2)$ функция $u \in W^{2, 1}_q(K)$ при некотором $M \in \mathbb{R}$ удовлетворяет условиям
    $$
    \partial_t u-\Delta u \leqslant 0, \quad u \leq M, \quad u \not\equiv M.
    $$
    Тогда если для некоторого $t' \in (t_1, t_2]$ выполнено $u(x_1, t') = M$, то
    $$
        \partial_x u (x_1, t') \ne 0.
    $$
\end{prop}
Также мы будем использовать следующее очевидное утверждение.
\begin{prop}\label{lem:uInd}
     Рассмотрим прямоугольник $K = (x_1, x_2)\times (t_1, t_2)$, функцию $u \in W^{2, 1}_q(K)$, открытое множество $A \subset K$ и $M \in \mathbb{R}$,  удовлетворяюшие условию
     $$
     u(x, t) = M, \; \partial_x u(x, t) = 0, \quad (x, t) = \partial A \setminus \partial K.
     $$
     Тогда функция 
\begin{equation}\label{eq:vu}
     v(x, t) = \begin{cases}
         u(x, t), & \quad (x, t) \in A, \\
         M, & \quad (x, t) \in K \setminus A 
     \end{cases}     
\end{equation}
     удовлетворяет условиям
     $$
      v \in W^{2, 1}_q(K), \quad   \partial_t v - \Delta v = (\partial_t u - \Delta u)\chi_{A}.
     $$     
\end{prop}

Докажем следующее вспомогательное утверждение
\begin{lemma} \label{lem:uleqbeta}
Пусть $\varphi \in W^{2-2/q}_q(\Omega)$ удовлетворяет условию \eqref{eq:phix0}.  Если $u$ -- решение задачи \eqref{problem_P}, 
%с топологически нетрансверсальной начальной функцией $\varphi \in W^{2-2/q}_q$, 
удовлетворяет \eqref{eq:ualpha}, \eqref{eq:ub} то выполнено неравенство
$$
u(x, t) \leqslant \beta, \quad \mbox{$x \in \mathcal{I}$, $t \in (0, T_1)$.}
$$
\end{lemma}

\begin{proof}
    Предположим противное, для некоторых  $t_1 > 0, x_1 \in I$ выполнено $u(x_1, t_1) > \beta$. Рассмотрим $x_2 \in \overline{\mathcal{I}}$, $t_2 \in [0, t_1]$ такие что
\begin{equation}\label{eq:t1max}
    u(x_2, t_2) = \max_{x \in \overline{\mathcal{I}},\, t \in [0, t_1]} u(x, t) > \beta.
\end{equation}
    Из неравенств \eqref{eq:phix0}, \eqref{eq:ub} следует, что $t_2 \ne 0$, $x_2 \notin \partial \mathcal{I}$. Отметим, что если $u(x, t) > \beta$ то $\mathcal{H}(u(x, t)) = -1$, а значит для достаточно малого $\delta > 0$ выполнено
    $$
     \partial_t u - \Delta u = -1, \quad \mbox{$x \in [x_2-\delta, x_2+\delta]$,\ $t \in  [t_2-\delta, t_2]$},
    $$
    а следовательно по предложению \ref{lem:max} $\max_{x \in [x_2-\delta, x_2+\delta], t \in  [t_2-\delta, t_2]} u(x, t)$ достигается при $x = x_2 \pm \delta$ или $t = t_2-\delta$, что противоречит \eqref{eq:t1max}.
\end{proof}

\begin{proof}[Доказательство Теоремы \ref{thm:noFB}]
    В силу \eqref{eq:r} для $x \in \mathcal{I}$, удовлетворяющих $r(x) < T_1$, выполнено
\begin{equation}\label{eq:u0}
    u(x, r(x)) = \beta.
\end{equation}
Из Леммы \ref{lem:uleqbeta} следует, что точка $(x, r(x))$ является локальным максимумом функции $u$ и следовательно 
\begin{equation}\label{eq:ux0}
    \partial_x u(x, r(x)) = 0.
\end{equation}

\textbf{Пункт $(i)$.} Из непрерывности $r(x)$ следует, что прямоугольник $K = (a, b) \times (0, T_1)$, функция $u$ и множество $A = \Int H^- \cap K$ и $M = \beta$ удовлетворяют условиям предложения \ref{lem:uInd}. Следовательно для функции $v$, заданной равенством \eqref{eq:vu}, выполнены неравенства
    $$
    v(x, t) \leqslant M, \quad \partial_tv - \Delta v \leqslant 0,
    $$
    а значит, она  удовлетворяет условиям предложения \ref{lem:max}. Предположим, что для некоторых $x_1, x_2 \in (a, b)$ выполнено $r(x_1), r(x_2) \in (0, T_1)$, и $r(x_1) > r(x_2)$. Тогда $\{x_1\} \times (r(x_2), r(x_1)) \subset H^-$. Из непрерывности функции $r$ на $(a, b)$ следует, что множество $B = \Int H^- \cap \left((a, b)\times (0, r(x_1))\right)$ не пусто. Из предложения \ref{lem:max} следует, что 
    $$
    v(x, t) = M, \quad t < r(x_1), \quad x \in (a, b). 
    $$
    В частности, $u|_B = v|_B \equiv M$, а значит, $(\partial_t u - \Delta u)|_B = 0$, что противоречит \eqref{eq:H0pm}. Пункт $(i)$ доказан.

\textbf{Пункт $(ii)$.} Предположим противное. Не умаляя общности, для некоторого $x_1 \in (a, b)$ выполнено неравенство 
\begin{equation}\label{eq:t0}
t_1 = \limsup_{x \to x_1-}r(x) > \limsup_{x \to x_1+}r(x) = \tilde{t}_1.    
\end{equation}
%Обозначим $t_1 = \limsup_{x \to x_1-}r(x)$. 
Из непрерывности $u$ и $\partial_xu$ и равенств \eqref{eq:u0}, \eqref{eq:ux0} следует, что 
\begin{equation}\label{eq:ux0t0}
u(x_1, t_1) = \beta, \quad \partial_xu(x_1, t_1) = 0.
\end{equation}
Из \eqref{eq:t0} cледует, что существуют такие $x_2 \in (x_1, b)$ и $t_2 \in (\tilde{t}_1, t_1)$, что
$$
r(x) < t_2, \quad x \in [x_1, x_2],
$$
а значит, $(x_1, x_2) \times (t_2, t_1] \subset H^-$. Рассмотрим два случая.

\textbf{Случай 1.} Для некоторой точки $(x', t') \in (x_1, x_2) \times (t_2, t_1]$ выполнено равенство $u(x', t') = \beta$.
Тогда  из предложения \ref{lem:max} аналогично доказательству пункта $(i)$ следует, что
$$
u = \beta \quad \mbox{ на $K = (x_1, x_2)\times(t_2, t')$},
$$
а значит, $(\partial_t u - \Delta u)|_{K} = 0$, что противоречит \eqref{eq:H0pm}.

\textbf{Случай 2.} Выполнены неравенства
$$
u(x, t) < \beta \quad \mbox{на $(x_1, x_2)\times(t_2, t_1]$}.
$$
Тогда из равенств \eqref{problem_P}, \eqref{eq:H0pm}, \eqref{eq:ux0t0} следует, что функция $u|_{(x_1, x_2)\times(t_2, t_1]}$ удовлетворяет условиям предложения \ref{lem:hopf}. Следовательно $\partial_x u(t_1, x_1) < 0$, что противоречит \eqref{eq:ux0t0}.

Пункт $(ii)$ доказан.

Докажем последнее утверждение теоремы. Обозначим 
$$
A = \Int(\{(x, t): x \in (a, b), t \geqslant r(x)\}), 
$$ 
    $$
    A_{t} = \{x \in (a, b): (x, t) \in A\}.
    $$
    Отметим, что для любого $t_0 \in (0, T)$ множество $A_{t_0}$ открыто. 
    Поэтому оно является  объединением не более чем счетного количества  интервалов:
    $$
    A_{t_0} = \bigcup_i (x_i, y_i).
    $$
    Покажем, что если $x_i \ne a$, то 
    \begin{equation}\label{eq:uxit0}
    u(x_i, t_0) = \beta, \quad \partial_x u(x_i, t_0) = 0.
    \end{equation}
    Действительно, поскольку $(x_i, y_i) \times \{t_0\} \subset A$, то
    $$
    r(x) < t_0, \quad x \in (x_i, y_i),
    $$
    следовательно $\limsup_{x \to x_i+}r(x) \leqslant t_0$. Рассмотрим 2 случая.

\noindent \textbf{Случай 1.} $\limsup\limits_{x \to x_i+}r(x) = t_0$. Из равенств \eqref{eq:u0}, \eqref{eq:ux0} ввиду непрерывности $\partial_x u$ следует \eqref{eq:uxit0}.
    
\noindent \textbf{Случай 2.} $\limsup\limits_{x \to x_i+}r(x) = t_1 < t_0$. Из пункта $(ii)$ следует, что $\limsup\limits_{x \to x_i-}r(x) = t_1$, а значит существует такое $\varepsilon > 0$, что 
    $$
    r(x) < t_0, \quad x \in (x_i -\varepsilon, x_i + \varepsilon),
    $$
    следовательно $x_i \in A_{t_0}$, противоречие. 
    
    Таким образом, \eqref{eq:uxit0} доказано.
    
    Аналогично можно показать, что для $y_i \ne b$ выполнено 
    \begin{equation*}
    u(y_i, t_0) = \beta, \quad \partial_x u(y_i, t_0) = 0,
    \end{equation*}
    и таким образом 
    \begin{equation*}\label{eq:uA0}
    u(x, t_0) = \beta, \quad x \in \partial A_{t_0} \setminus \{a, b\}.
    \end{equation*}

    Пусть теперь $x_0$ -- какая-нибудь точка из интервала $(a, b)$. Заметим, что если $(x_0,t_1) \in H^{-}$, то $\{x_0\} \times (t_1, T_1) \subset H^-$. Поэтому для любого $x_0 \in (a, b)$ найдется $t_0 \geq 0$ такое, что $A \cap \{x = x_0\} = \{x_0\} \times (t_0, T_1)$.

    %Для любого $x_0 \in (a, b)$ выполнено $A \cap \left(\{x_0\}\times(0, T_1)\right) = \{x_0\} \times (t_0, T_1)$ для некоторого $t_0 \geqslant 0$. 
    
    Поскольку  $\{x_0\} \times (t_0, T_1) \subset A$, то $\limsup_{x \to x_0} r(x) \leqslant t_0$. Поскольку $(t_0, x_0) \notin A$, то $\limsup_{x \to x_0} r(x) \geqslant t_0$. А значит, $\tilde{r}(x_0) = \limsup_{x \to x_0} r(x) = t_0$ и следовательно
    $$
    A = \{(x, t): x \in (a, b), t \geqslant \tilde{r}(x)\},
    $$
    и из \eqref{eq:u0}, \eqref{eq:ux0} получаем ввиду непрерывности $\partial_x u$
    \begin{equation}\notag
    u(x, \tilde{r}(x)) = \beta, \quad \partial_x u(x, \tilde{r}(x)) = 0, \quad x \in (a, b).        
    \end{equation}    
    Аналогично пункту $(i)$ отсюда следует, что для некоторого $\tau \in (0, T_1)$ выполнено
    $$
    \tilde{r}(x) = \tau, \quad x \in (a, b), \quad \mbox{и} \quad u(x, \tau) = \beta, \quad x \in (a, b),
    $$
и последнее утверждение доказано.
\end{proof}

\vspace{0.4cm}
Авторы признательны А.И.~Назарову за чрезвычайно полезные обсуждения и замечания, которые способствовали существенному улучшению изложения. 
%существенно улучшили качество статьи.

\bibliography{Bibliography-hysteresis}

\begin{thebibliography}{10}

\bibitem{KraP1983}
{М.А.~Красносельский, А.В.~Покровский}.
\newblock {\em Системы с гистерезисом}.
\newblock М.: Наука, 1983.

\bibitem{Vis1994}
{A.~Visintin}.
\newblock {\em Differential Models of Hysteresis.}
\newblock Springer-Verlag. Berlin --- Heidelberg, 1994.

\bibitem{BroSpre1996}
{M.~Brokate, J.~Sprekels}.
\newblock {\em Hysteresis and Phase Transitions.}
\newblock Springer, Berlin, 1996.

\bibitem{HJ1980}
{F.C.~Hoppensteadt, W.~J\"{a}ger}.
\newblock Pattern formation by bacteria.
\newblock In {\em Biological growth and spread ({P}roc. {C}onf., {H}eidelberg, 1979)}, volume~38 of {\em Lecture Notes in Biomath.}, pages 68--81. Springer, Berlin-New York, 1980.

\bibitem{HJP1984}
{F.C.~Hoppensteadt, W.~ J\"{a}ger, C.~P\"{o}ppe}.
\newblock A hysteresis model for bacterial growth patterns.
\newblock In {\em Modelling of patterns in space and time ({H}eidelberg, 1983)}, volume~55 of {\em Lecture Notes in Biomath.}, pages 123--134. Springer, Berlin, 1984.

\bibitem{AU15}
{D.E.~Apushkinskaya, N.N.~Uraltseva}.
\newblock Free boundaries in problems with hysteresis.
\newblock {\em Philos. Trans. Roy. Soc. A}, 373(2050):20140271, 10, 2015.

\bibitem{CGT2016}
{M.~Curran, P.~Gurevich, S.~Tikhomirov}.
\newblock Recent advance in reaction-diffusion equations with non-ideal relays.
\newblock In {\em Control of self-organizing nonlinear systems}, Underst. Complex Syst., pages 211--234. Springer, [Cham], 2016.

\bibitem{Vis1986}
{A.~Visintin}.
\newblock Evolution problems with hysteresis in the source term.
\newblock {\em SIAM J. Math. Anal.}, 17(5):1113--1138, 1986.

\bibitem{Alt1985}
{H.W.~Alt}.
\newblock On the thermostat problem.
\newblock {\em Control Cybernet.}, 14(1-3):171--193, 1985.

\bibitem{Kop07}
{J.~Kopfov\'a}.
\newblock Nonlinear semigroup methods in problems with hysteresis.
\newblock {\em Discrete Contin. Dyn. Syst.}, pages 580--589, 2007.

\bibitem{Vis2014}
{A.~Visintin}.
\newblock Ten issues about hysteresis.
\newblock {\em Acta Appl. Math.}, 132:635--647, 2014.

\bibitem{GShT13}
{P.~Gurevich, R.~Shamin, S.~Tikhomirov}.
\newblock Reaction-diffusion equations with spatially distributed hysteresis.
\newblock {\em SIAM J. Math. Anal.}, 45(3):1328--1355, 2013.

\bibitem{GT12}
{P.~Gurevich, S.~Tikhomirov}.
\newblock Uniqueness of transverse solutions for reaction-diffusion equations with spatially distributed hysteresis.
\newblock {\em Nonlinear Anal.}, 75(18):6610--6619, 2012.

\bibitem{AU15a}
{D.E.~Apushkinskaya, N.N.~Uraltseva}.
\newblock On regularity properties of solutions to the hysteresis-type problem.
\newblock {\em Interfaces Free Bound.}, 17(1):93--115, 2015.

\bibitem{GT2017}
{P.~Gurevich, S.~Tikhomirov}.
\newblock Rattling in spatially discrete diffusion equations with hysteresis.
\newblock {\em Multiscale Model. Simul.}, 15(3):1176--1197, 2017.

\bibitem{GT2018}
{P.~Gurevich, S.~Tikhomirov}.
\newblock Spatially discrete reaction-diffusion equations with discontinuous hysteresis.
\newblock {\em Ann. Inst. H. Poincar\'{e} C Anal. Non Lin\'{e}aire}, 35(4):1041--1077, 2018.

\bibitem{LSU67}
{О.А.~Ладыженская, В.А.~Солонников, Н.Н.~Уральцева}.
\newblock {\em Линейные и квазилинейные уравнения параболического типа}.
\newblock М.: Наука, 1967.

\bibitem{KA1984}
{Л.В.~Канторович, Г.П.~Акилов}.
\newblock {\em Функциональный анализ}.
\newblock М.: Наука, 1984.
\newblock Изд. 3-е, перераб.

\bibitem{Shu2003}
{М.А.~Шубин}.
\newblock {\em Лекции об уравнениях математической физики}.
\newblock М.: МЦНМО, Изд. 2е, исправ. edition, 2003.

\bibitem{krylov1985}
Н.В. Крылов.
\newblock {\em Нелинейные эллиптические и параболические уравнения второго порядка}.
\newblock Наука, 1985.

\bibitem{AN2019}
{D.E.~Apushkinskaya, A.I.~Nazarov}.
\newblock On the boundary point principle for divergence-type equations.
\newblock {\em Atti Accad. Naz. Lincei Rend. Lincei Mat. Appl.}, 30(4):677--699, 2019.

\bibitem{N2012}
{A.I.~Nazarov}.
\newblock A centennial of the {Z}aremba-{H}opf-{O}leinik lemma.
\newblock {\em SIAM J. Math. Anal.}, 44(1):437--453, 2012.

\end{thebibliography}

\noindent
Российский университет 

\noindent
дружбы народов, г. Москва, 

\noindent
\textit{E-mail:} apushkinskaya@gmail.com
\vspace{0.4cm}

\noindent
Pontifícia Universidade Católica do Rio

\noindent
de Janeiro - PUC-Rio, г. Рио де Жанейро

\noindent
\textit{E-mail:} sergey.tikhomirov@gmail.com
\vspace{0.4cm}

\noindent
Санкт-Петербургский государственный 

\noindent
университет, г. Санкт-Петербург

\noindent
\textit{E-mail:} uraltsev@pdmi.ras.ru

\newpage
Apushkinskaya D.E., Tikhomirov S.B., Uraltseva N.N.

\textbf{Properties of the phase boundary in the parabolic problem with hysteresis} \vspace{0.3cm}

\noindent
\textit{Keywords:} hysteresis, parabolic equation, phase boundary, transversality, existence theorem

\noindent
\textit{abstract:} We study solutions of parabolic equations with a discontinuous hysteresis operator, described by a free interface boundary. It is established that for spatially transverse initial data from the space $W^{2-2/q}_q$ with $q > 3$, there exists a solution in the space $W^{2,1}_q$, where the interface boundary exhibits Holder continuity with an exponent of $1/2$. Furthermore for initial data from the space $W^2_\infty$, it is proven that the interface boundary satisfies the Lipschitz condition. It is shown that for non-transversal initial data, solutions with an interface boundary do not exist.

\vspace{0.3cm}

\noindent
RUDN University, Moscow 

\noindent
\textit{E-mail:} apushkinskaya@gmail.com
\vspace{0.4cm}

\noindent
Pontifícia Universidade Católica do Rio

\noindent
de Janeiro - PUC-Rio, Rio de Janeiro

\noindent
\textit{E-mail:} sergey.tikhomirov@gmail.com
\vspace{0.4cm}

\noindent
St. Petersburg State University, 

\noindent
St. Petersburg

\noindent
\textit{E-mail:} uraltsev@pdmi.ras.ru
\end{document}